\documentclass[12pt, a4paper]{article}
\usepackage{graphicx} 

\usepackage{tikz}

\usepackage{amsmath,amsthm,amscd,amssymb,eucal,mathrsfs}
\usepackage{amsfonts}
\usepackage{latexsym}
\usepackage{graphicx} 
\usepackage{mathtools}
\usepackage{multicol}
\usepackage{dsfont}
\usepackage{natbib}
\usepackage[all]{xy}
\usepackage{multirow}
\usepackage{color}
\usepackage{mathtools}
\usepackage[hidelinks]{hyperref}
\usepackage[all]{xy}
\usepackage{tcolorbox}
\usepackage{fullpage}

\newtheorem{thm}{Theorem}[section]
\newtheorem{prop}[thm]{Proposition}

\newtheorem{cor}[thm]{Corollary}
\newtheorem{remark}[thm]{Remark}
\newtheorem{definition}[thm]{Definition}

\newtheorem{Proposition}[thm]{Proposition}

\newtheorem*{Satz*}{Satz}

\newtheorem{Lemma}[thm]{Lemma}

\newcommand{\mathset}[1]{{\left\{#1\right\}}}
\newcommand{\absolute}[1]{\left\lvert#1\right\rvert}

\DeclareMathOperator{\Spec}{Spec}

\DeclareMathOperator{\supp}{supp}

\title{Hearing the Serre invariant of a compact $p$-adic analytic manifold}
\author{Patrick Erik Bradley
and \'Angel Mor\'an Ledezma
\\
Karlsruhe Institute of Technology
\\
Institute of Photogrammetry and Remote Sensing 
\\
and 
\\
Geodetic Institute
\\
Englerstr.\ 7
\\
76131 Karlsruhe
\\
Germany
}
\date{\today}

\begin{document}

\maketitle

\begin{abstract}
Using a previous novel way of defining kernel functions for Laplacian integral operators on a compact $p$-adic analytic manifold $X$, one such operator $\Delta_0^s$ with $s\in\mathds{R}$ is applied to hearing the Serre invariant  $i(X)$ by showing that a wavelet eigenvalue is always congruent to $i(X)$ modulo $q-1$, where $q$ is the cardinality of the residue field $k$ attached to a $p$-adic number field $K$. It is shown how the number of $k$-rational points of the special fibre of the N\'eron model of an elliptic curve defined over $K$
relates to the wavelet spectrum 
of $\Delta_0^s$, and this then leads to the realisation that the Serre invariant $i(E(X))$ in the case of an elliptic curve $E$ with split multiplicative reduction vanishes modulo $q-1$.
\end{abstract}


\section{Introduction}

The endeavour of enabling diffusion processes on $p$-adic domains other than $K^n$ or compact open subsets of $K^n$, where $K$ is a non-archimedean local field is undertaken since \cite{DiffMfp} by using an atlas on a compact $p$-adic analytic domain having a connected nerve complex. That work is inspired by the various methods employed in the case of local fields $K$\cite{Taibleson1975,VVZ1994,Kochubei},  the $n$-dimensional vector space $K^n$ \cite{RW2023}, 
finite-dimensional $K$-vector spaces \cite{PRSWY2024}, $p$-adic balls \cite{Kochubei2018,PW2024}, and finite disjoint  unions of $p$-adic discs \cite{ZunigaNetworks}. These works all use more or less some variation of the Vladimirov-Taibleson operator from \cite{Taibleson1975,VVZ1994}.
The idea for extending these to spaces which are locally an open piece of $K^n$ leads to the study of diffusion operators on Mumford curves \cite{brad_HeatMumf,brad_thetaDiffusionTateCurve,HearingGenusMumf}, also in such a way that topological and geometric properties of their underlying graphs can be extracted from the spectrum of this operators.
In \cite{BL_shapes_p}, the graph can be reconstructed from such a spectrum.
\newline

The new idea of \cite{DiffMfp} is to use the local structure of a $p$-adic analytic manifold $X$ coming from charts of an atlas on $X$. This is forced by the fact that e.g.\ the $K$-rational points of a projective algebraic variety form a compact subset of some projective $n$-space which can only locally be embedded into some affine space over $K$.
And if the underlying sets of the charts of a given atlas of $X$ form a covering such that the passage through overlaps of these sets is possible between any two points on the manifold, i.e.\ when the nerve complex associated with the atlas is connected, then, under some mild condition, a kernel function for a Laplacian integral operator can be defined which locally on small balls is a restricted Vladimirov-Taibleson operator. This allowed to extract the number of points of an elliptic curve with good reduction from the spectrum of such an operator \cite[Theorem 4.4]{DiffMfp}. 
A desideratum is the case of time-dependent diffusion on $p$-adic manifolds, extending the cases of \cite{NonAutonomous,Ledezma-Energy}, as well as the study of boundary problems as in \cite{ellipticBVP}.
\newline

The goal of this article is to use a slight modification of the operator $\Delta^s$ ($s\in\mathds{R}$) from \cite{DiffMfp} in order to obtain the Serre invariant of a compact $p$-adic manifold, and to apply this to  elliptic curves over $K$. Of particular interest here is  the manifold of $K$-rational points of a Tate elliptic curve defined over $K$.
These form a special class of elliptic curves with split multiplicative reduction. The Serre invariant $i(X)$ of a $p$-adic analytic manifold $X$ is given as the volume of $X$ with respect to the measure obtained from a nowhere vanishing differential $n$-form on $X$, where $n$ is the dimension of $X$. And the result by J.-P.\ Serre in \cite{Serre1965} is that this invariant $i(X)$ modulo $q-1$ represents a minimal way of writing $X$ as a disjoint union of $i(X)$ copies of the ball $O_K^n$, and all other ways of doing this are congruent $i(X)$ modulo $q-1$. Here, the number $q$ signifies the cardinality of the residue field $k$ associated with the local field $K$.
The significance of being able to  hear the Serre invariant is contained in the task of inferring structure from Laplacian diffusion or the spectra of Laplacian diffusion operators. In the case of compact $p$-adic analytic manifolds, this means that the space  looks like a $p$-adic ball with $i(X)+1$ holes modulo $q-1$. The main result of this article is that the wavelet eigenvalues for the operator $\Delta_0^s$ are all congruent to the Serre invariant $i(X)$ modulo $q-1$.
\newline

The following Section 2 introduces Serre's result about his invariant for compact $p$-adic manifolds $X$ and at the same time explains integration on these with the help of the measure $\mu_\omega$ obtained from nowhere vanishing differential $n$-form $\omega$, where $n$ is the dimension of $X$.  Section 3 introduces integral operators, as defined in \cite{DiffMfp}, using the distance function obtained by $\mu_\omega$ and an atlas having connected nerve complex in order to build kernel functions. The operator $\Delta_0^s$ considered here is different from $\Delta^s$ introduced in \cite{DiffMfp}. In this section, the eigenvalues associated with wavelets supported in $X$ are calculated. Section 4 is devoted to hearing the Serre invariant $i(X)$ of a compact $p$-adic analytic manifold from wavelet eigenvalues $\lambda_\psi\in\Spec(\Delta_0^s)$. The first result is
\[
i(X)\equiv \lambda_\psi\mod q-1\,,
\]
stated as Theorem 4.2. Using the theory of N\'eron models, this is used to hear the Serre invariant of an elliptic curve $E(K)$ over a sufficiently large $p$-adic number field, as 
\[
\mu_\omega(E(K))=(E(K):E_0(K))\frac{\absolute{\overline{E}_0(k)}}{q}\,,
\]
where $E_0(K)$ consists of the subgroup of points in $E(K)$ specialising to non-singular points in the reduction curve, $\overline{E}(k)$ is the identity component of the special fibre of the N\'eron model $\mathcal{E}$, and $q$ is the cardinality of the residue field $k$ of $K$. This is Proposition 4.4, which is used to prove the final result:
\[
\lambda_\psi\equiv 0\mod q-1
\]
in the case of a Tate elliptic curve $E(K)$ over $K$. This is Corollary 4.6.

\section{Compact $p$-adic analytic manifolds}

Let $K$ be a non-archimedean local field, i.e.\ $K$ is a locally compact field w.r.t.\ an absolute value coming from a discrete valuation. The residue field of $K$ is finite, and its number of elements is written as $q$, and is of the form
\[
q=p^f
\]
with $p$ a prime number and $f$ a positive natural number. In the case that $K=\mathds{Q}_p$, the field of $p$-adic numbers, it follows that $q=p$, i.e.\ $f=1$.  In the following, we will speak of $p$-adic analytic manifolds, instead of $K$-analytic ones, just like Serre himself did in \cite{Serre1965}.
\newline

An introduction to $p$-adic analytic manifolds can be found e.g.\ in \cite{Serre1992,WeilAAG,IgusaLocalZeta,Schneider2011}.
\newline

From \cite[Chapters 7 and 8]{Schneider2011}, learn that a $p$-adic chart for a Hausdorff topological space $M$ is a triple $(U,\phi,K^n)$ such that
$U\subseteq M$ is open, and $\phi\colon U\to K^n$ is a homeomorphism onto its image $\phi(U)$, and this image is open in $K^n$. Two charts $(U,\phi,K^n)$, $(V,\psi,K^m)$ are compatible, if the maps
\[
\xymatrix{
\phi(U\cap V)\ar@<.5ex>[r]^{\psi\circ\phi^{-1}}&\psi(U\cap V)\ar@<.5ex>[l]^{\phi\circ\psi^{-1}}
}
\]
are both locally analytic.
An atlas of $M$ is a family of compatible charts whose underlying sets form a cover of $M$. Two atlantes $\mathcal{A},\mathcal{B}$ of $M$ are equivalent, if $\mathcal{A}\cup\mathcal{B}$ is an atlas of $M$. An atlas $\mathcal{A}$ of $M$ is maximal, if any atlas equivalent with $\mathcal{A}$ is contained in $\mathcal{A}$. One observes that equivalence of atlantes is an  equivalence relation, and that each equivalence class of atlantes of $M$ contains a maximal atlas \cite[Remark 7.2]{Schneider2011}. A $p$-adic analytic manifold is a Hausdorff space $X$ equipped with a maximal atlas $\mathcal{A}$. If all charts in $\mathcal{A}$ are of equal dimension $n$, then $X$ is called a $p$-adic analytic $n$-manifold.
 \newline
 
Notice that here, the $p$-adic $n$-manifold $X$ is assumed to be compact and non-empty. 
Recall the following results by J.-P.\ Serre whose detailed proofs can be found in Igusa's book \cite[Chapter 7.5]{IgusaLocalZeta}, and named  Theorem \ref{SerreTheorem1} and Theorem \ref{Serre1}, below.

\begin{thm}[Serre]\label{SerreTheorem1}
The following statements hold true:
\begin{enumerate}
\item Every compact $p$-adic analytic $n$-manifold $X$ is isomorphic to the disjoint union $r.O_K^n$ of $r$ copies of $O_K^n$ for some integer $r\ge1$.
\item The manifolds $r.O_K^n$ and $r'.O_K^n$ are isomorphic, if and only if 
\[
r\equiv r'\mod q-1\,,
\]
where $q$ is the cardinality of the residue field of $K$.
\end{enumerate}
\end{thm}

\begin{proof}
\cite[Th\'eor\`eme (1)]{Serre1965}.
\end{proof}

The residue class $r\mod q-1$ associated with a compact $p$-adic analytic manifold $X$ is called the \emph{Serre invariant} of $X$, and is denoted as $i(X)$.
\newline

Denote the Haar measure on $K^n$ as 
\[
\absolute{dx}=\absolute{dx_1}\wedge\cdots\wedge \absolute{dx_n}
\]
with coordinate Haar measures $\absolute{dx_i}$,
and normalise it by setting
\[
\int_{O_K}\absolute{dx_i}=1
\]
for $i=1,\dots,n$. Given an analytic differential $n$-form $\omega$ on an analytic $n$-manifold $X$, it can locally on a chart $\phi\colon U\to K^n$ be written in a sloppy way as
\begin{align}\label{sloppyDifferential}
\omega|_U:=\phi_*\omega=f_U\,dx_1\wedge\dots\wedge dx_n=f_U\,dx
\end{align}
for some analytic function $f_U\colon \phi(U)\to K$. More or less sloppy notations for (\ref{sloppyDifferential}) are these:
\begin{align*}
\omega|_{U}(x)&=&\,\text{[sloppy notation]}
\\
&=\phi_*\omega(x)&\text{[good notation]}
\\
&=\omega_U(\phi(x))&\text{[meaningful notation]}
\\
&=f_U(\phi(x))\,dx_1\wedge\dots\wedge dx_n&\text{[better understandable notation]}
\\
&=f_U(\phi(x))\,dx&\text{[understandable shorthand notation]}
\\
&=f_U(x)\,dx&\text{[other sloppy notation]}\,,
\end{align*}
where one should notice that $\omega_U$ is the differential form on $\phi(U)\subseteq K^n$ locally representing $\omega$ on the chart $(U,\phi)$. That is the forerunner to the idea of a sheaf in action!
\newline

This now yields a measure
\begin{align}\label{localMeasure}
\absolute{\omega|_U}(A):=\int_{\phi(A)}\absolute{f_U(x)}\absolute{dx}
\end{align}
for $A\subset U$ such that $\phi(A)\subset K^n$ is measurable w.r.t.\ $\absolute{dx}$.
Notice that in (\ref{localMeasure}), the expression $f_U(x)$ is less sloppy than one would believe from the above, because now the variable $x$ runs through $\phi(A)$. But of course, one can sloppify the integral to
\[
\int_A\absolute{f_U(x)}\absolute{dx}\,,
\]
if one is not able to refrain from the urge to do so.
\newline

Since the locally defined forms $\omega|_U$ (on $\phi(U)$!) ``glue'' together to  the given differential $n$-form $\omega$ on $X$,  one obtains a  measure $\absolute{\omega}$ on $X$ minus the vanishing set of $\omega$,
 locally on a chart 
 given by (\ref{localMeasure}).
This construction can be found in \cite[Chapter 2.2]{WeilAAG} and \cite[Chapter 7.4]{IgusaLocalZeta}. 

\begin{thm}[Serre]\label{Serre1} 
Let $X$ be a compact $p$-adic analytic $n$-manifold. Then 
\begin{enumerate}
\item There exists a nowhere vanishing analytic differential $n$-form on $X$.
\item If $\omega$ is a nowhere vanishing analytic differential $n$-form on $X$, then 
\[
i(X)\equiv a\equiv \int_X\absolute{\omega(x)}\mod q-1\,,
\]
where the identity
\[
\int_X\absolute{\omega}=\frac{a}{q^b}
\]
holds true for some $a,b\in\mathds{N}$.
\end{enumerate}
\end{thm}

\begin{proof}
\cite[Th\'eor\`eme (2)]{Serre1965}.
\end{proof}

An immediate consequence of Theorem \ref{Serre1} is that a compact $p$-adic analytic manifold $X$ always has an everywhere defined measure. Write this measure also as
\[
\mu_\omega(A)=\int_A\absolute{\omega}
=\int_X 1_A\,\absolute{\omega}
\]
for $A\subset X$ measurable w.r.t.\ to the measure $\absolute{\omega}$ given by a nowhere vanishing differential $n$-form $\omega$ on  $X$, and using the indicator function $1_A$ of $A$.

\section{Operators on compact $p$-adic analytic manifolds}

The operators on a $p$-adic analytic manifold $X$ constructed in \cite{DiffMfp} use a certain kind of atlas $\mathcal{A}$ on $X$, for which the nerve complex $N(\mathcal{A})$ is connected, and the so-called \emph{equalising property} holds true.
We will briefly explain these notions here.
\newline

The first thing to remark is that an analytic structure, i.e.\ an equivalence class of atlantes, does in general not produce a unique metric on a manifold. This is well-known in the case of Riemannian manifolds. Another issue is given by 
Theorem \ref{SerreTheorem1}: namely, that any sufficiently fine atlas is a collection of disjoint charts. This makes it difficult to define a transition rate function in a meaningful way. However, an atlas with sufficient overlaps produces a simplicial structure with the connectedness property---the nerve complex $N(\mathcal{A})$ produced from the open sets in the charts covering $X$. It is assumed that distinct charts $(U,\phi_U),(V,\phi_V)$ always have distinct open sets $U\neq V$. The nerve complex is defined by saying that its $k$-simplices are given by the non-empty intersection of
 $k$ open sets appearing in charts of $\mathcal{A}$. The first requirement is that $N(\mathcal{A})$ be a connected simplicial complex.
The next observation is that the transition maps
\[
\tau_{\phi_U,\phi_V}\colon\phi_U(U\cap V)\to\phi_V(U\cap V)
\]
on intersections $U\cap V\neq\emptyset$ for charts $(U,\phi_U),(V,\phi_V)\in\mathcal{A}$ are bi-analytic, and any finite atlas on a compact $p$-adic analytic manifold can be replaced by one such that all the sets occurring in its charts are the same, but the chart maps are such that $\tau_{\phi_U,\phi_V}$ takes sufficiently small balls to balls of equal radius. This is the equalising property of an atlas. Any compact $p$-adic analytic manifold has a finite equalising atlas, cf.\ \cite[Theorem 2.6]{DiffMfp}.
\newline

In this way, the notion of a ball on a compact $p$-adic analytic manifold, and also its radius, becomes well-defined. However, the radius itself is not used, but it is the measure of a ball which is of interest. So, if two points $x,y\in X$ lie in a common ball, the join
\[
x\wedge y=\bigcap\limits{B\ni x,y}B\,,
\]
where $B$ runs through all balls of $X$ containing $x$ and $y$, can be defined, as well as the simplex 
\[
\sigma(x)=\text{the highest-dimensional simplex in $N(\mathcal{A})$ containing $x\in X$}
\]
can be defined.
From this, the distance $d_g(x,y)$ for $x,y\in X$ is given as
\[
d_g(x,y)=\begin{cases}
\mu_\omega(x\wedge y),&\exists\,x\wedge y
\\
\min\limits_{\gamma\colon\sigma(x)\leadsto\sigma(y)}\mu_X(U_\gamma(x,y)),&\nexists\,x\wedge y\,,
\end{cases}
\]
where $\gamma$ runs through all simplicial paths in $N(\mathcal{A})$  between $x$ and $y$, and 
\[
U_\gamma=\bigcup\limits_{\sigma\in\gamma}\sigma\subset X
\]
is the union of the sets associated with the simplices $\sigma$ of $N(\mathcal{A})$ along the path $\gamma$. The distance $d_g$ is called the \emph{$p$-adic geodetic distance} on $X$. It was used in \cite{DiffMfp} to construct the kernel function of an operator called \emph{$p$-adic Laplace-Beltrami} operator $\Delta^s$ on the function space $\mathcal{D}(X)$ of locally constant functions $X\to\mathds{C}$. Its definition is
\[
\Delta^sf(x)=\int_X d_g(x,y)^{-s}(f(x)-f(y))\,d\mu_\omega(y)
\]
with $s\in\mathds{R}$. Since the nerve complex $N(\mathcal{A})$ is connected, the distance function $d_g$ can now be used in order to interpret $d_g(x,y)^{-s}$ as a transition rate for jumping between two points $x,y\in X$.
\newline 

For the purpose of this article define the following kernel function 
\[
k_0\colon X\times X,\;
(x,y)\mapsto 
\begin{cases}
\mu_\omega(x\wedge y)^{-s},& \exists\,x\wedge y\in X
\\
1,&\text{otherwise}
\end{cases}
\]
on the manifold $X\times X$. 
The corresponding integral operator
on the space $\mathcal{D}(X)$ of locally constant functions $X\to\mathds{C}$: 
\[
\Delta_0^\alpha u(x)
=\int_X k_0(x,y)(u(x)-u(y))\,d\mu_\omega(y)
\]
for  $u\in\mathcal{D}(X)$
is called  a \emph{Vladimirov-Z\'u\~niga} operator.
Observe that when restricted to the space $\mathcal{D}(B)$ for $B\subset X$ a ball, the operators $\Delta^s,\Delta_0^s$ coincide with the Taibleson operator as defined in \cite{Taibleson1975} which in turn is a generalisation of the Vladimirov operator from \cite{VVZ1994}.
\newline

A \emph{Koyzrev wavelet} on $K^n$ is a locally constant function of the form
\[
\psi_{B(a),j}(x)=\mu(B(a))^{-\frac12}\chi\left(\pi^{d-1}\eta(j)x\right)1_{B(a)}(x)\,,
\]
where $\mu$ denotes the product Haar measure on $K^n$,
$a\in K^n$, $B(a)\subset K^n$ is a ball centred in $a$, $j\in \left[\left(O_K/(\pi)O_K\right)^\times\right]^n$, 
\[
\eta\colon \left(O_K/\pi O_K\right)^n\to K
\]
a lift of the canonical projection $O_K\to O_K/\pi O_K$, where $\pi$ is the uniformiser of $K$, and 
\[
\chi\colon K^n\to S^1
\]
a fixed complex-valued unitary additive character of $K$. 
In \cite[Definition 3.7]{DiffMfp}, the notion of wavelet is extended to a $p$-adic analytic manifold $X$ by a function supported on a ball $B(a)\subset X$ centred in $a\in X$ and requiring it to be a Kozyrev wavelet supported on the ball $\phi_U(B(a))\subset K^n$ centred in $\phi_U(a)$ for any chart $(U,\phi_U)\in\mathcal{A}$ whose set $U$ contains $B(a)$.
\newline

\begin{remark}
In contrast to the case of $K^n$, it can happen in general $p$-adic manifolds that distinct balls have a genuine overlap, i.e.\ they are neither disjoint nor one contained in the other. For example, if $X$ is the projective line, then it is the union of the unit disc with a copy of itself given by the map $z\mapsto z^{-1}$, and their intersection is the unit sphere.
\end{remark}

\begin{prop}\label{Spectrum}
The 
 wavelets $\psi$ of $X$ are eigenfunctions of $\Delta_0^s$
with eigenvalue
\[
\lambda_{\psi}=
\int_{X\setminus U(x)}\,d\mu_\omega+
\int_{U(x)\setminus B}
\mu_\omega(x\wedge y)^{-s}\,d\mu_\omega(y)
+\mu_\omega(B)^{1-s}
\]
with $x\in\supp(\psi)=B\subseteq U(x)\subseteq X$, where $U(x)$ is the largest ball in $X$ containing $x\in X$, and with $s\in\mathds{R}$. 
\end{prop}

\begin{proof}
The important vanishing of the oscillatory integral 
\[
\int_X\psi\,\mu_\omega=0\,,
\]
is shown in \cite[Lemma 3.8]{DiffMfp}.
The eigenvalue calculation of
\cite[Theorem 3]{Kozyrev2004} can now be applied directly to this case and yields the asserted value for $\lambda_\psi$.
\end{proof}

\begin{remark}
In \cite[Lemma 4.9]{SchottkyInvariantDiffusion}, a similar eigenvalue calculation based on the method of \cite[Theorem 3]{Kozyrev2004} was used for a measure coming from an algebraic differential $1$-form on the $p$-adic points of a Mumford curve. 
\end{remark}

\begin{remark}
The operators $\Delta^s,\Delta_0^s$  can be viewed as  Parisi-Z\'u\~niga-type operators on a finite complete graph, similar to the operators in \cite{IndexTopo_p}.
\end{remark}

\section{Hearing the Serre number}

Let $X$ be a compact $p$-adic manifold with  a finite equitising atlas $\mathcal{A}$, such that the nerve complex $N(\mathcal{A})$ is connected, as before.

\subsection{The Serre Invariant of  compact $p$-adic analytic manifolds}

First, calculate the Serre invariant of a sphere  on $K^n$.

\begin{Lemma}\label{Sphere}
The Serre invariant of a sphere in $K^n$ vanishes modulo $q-1$.
\end{Lemma}

\begin{proof}
A sphere $S\subset K^n$ is a $p$-adic $n$-ball minus a proper maximal subball. Hence,
\[
i(S)\equiv q^n-1\equiv
(q-1)\sum\limits_{k=0}^{n-1}q^k\equiv 0\mod q-1
\]
as asserted.
\end{proof}

Using the operator $\Delta_0^s$ with $s\in\mathds{R}$ allows now to hear its Serre invariant of a compact $p$-adic analytic manifold $X$:

\begin{thm}\label{SerreInvariant}
Let $X$ be a compact $p$-adic analytic manifold,  $s\in\mathds{R}$, and let $\psi$ be a Kozyrev wavelet  supported in $X$. Then the corresponding eigenvalue $\lambda_\psi$ of $\Delta_0^s$ satisfies
\[
\lambda_\psi\equiv i(X)\mod q-1\,,
\]
i.e.\ the Serre invariant $i(X)$ of the  $(X,\mathcal{A})$ can be read off the spectrum of $\Delta^s$.
\end{thm}

\begin{proof}
Let $x\in B=\supp(\psi)\subseteq U(x)\subseteq X$, where $U(x)$ is the largest ball in $X$ containing $x\in X$.
According to Proposition \ref{Spectrum}, it holds true that
\begin{align*}
\lambda_\psi&\equiv 
i(X)-i(U(x)) + \int_{U(x)\setminus B}\mu_\omega(x\wedge y)^{-s}\,d\mu_\omega(y)
+i(B)^{1-s} 
\\
&\equiv i(X)- i(U(x))
+\int_{U(x)\setminus B}i(x\wedge y)^{-s}\,d\mu_\omega(y) +1  
\\
&\equiv i(X)-i(U(x))+
\sum\limits_{k=r}^t
i(S_k(x))^{1-s}+1
\\
&\equiv i(X)-i(U(x))
+0+1
\\
&\equiv i(X)
\mod q-1\,,
\end{align*}
where $S_k(x)$ is a sphere of radius $q^{-k}$ centred in $x\in X$, with $k=r,\dots,t\in\mathds{Z}$, such that these form a disjoint covering of the annulus $U(x)\setminus B$, and
because $i(U(x))\equiv i(B)\equiv1\mod q-1$.
This proves the assertion.
\end{proof}



\subsection{The Serre invariant of an  elliptic curve}

In order to compute the Serre invariant of the $K$-rational points $E(K)$ of an elliptic curve $E$ defined over the local field $K$, some facts about N\'eron models are needed. These can be found e.g.\ in \cite[Chapter IV]{Silverman1994}. 
The difference to the article \cite{DiffMfp} is that here, the operator $\Delta_0^s$ will be used, whereas it is the operator $\Delta^s$ in \emph{loc.\ cit.}
\newline

Recall from \cite[Chapter IV.5]{Silverman1994}:
\begin{definition}
A \emph{N\'eron} model for $E/K$ is a (smooth) group scheme $\mathcal{E}/O_K$ whose generic fibre is $E/K$, and which satisfies the universal property: Let $\mathcal{X}/R$ be a smooth $O_K$-scheme with generic fibre $X/K$, and let $\phi_K\colon X\to E$ be a rational map defined over $K$. Then there exists a unique $O_K$-morphism $\phi_{O_K}\colon \mathcal{X}\to\mathcal{E}$ extending $\phi_K$.
\end{definition}

The universal property ensures that there is an injective map
\[
\iota\colon\mathcal{E}(O_K)\to E(K)\,,
\]
and, after a sufficiently large extension of $K$, it may and will be assumed that $\iota$ is an isomorphism, i.e.\ that
\[
\mathcal{E}(O_K)=E(K)
\]
holds true. In this case, there is a reduction map
\[
\rho\colon E(K)\to\overline{E}(k)\,,
\]
where $\overline{E}$ is the special fibre of $\mathcal{E}$ defined over the residue field $k=O_K/\pi O_K$ of $K$, where $\pi\in O_K$ is the uniformiser of $K$.
\newline

Let $\omega\in \Omega^1(E/K)$ be a nowhere vanishing differential $1$-form on $E(K)$. Using the fact that $E$ is projective algebraic curve of genus $1$, allows to take an algebraic differential $1$-form for $\omega$, as they are non-vanishing. This then allows to produce  a measure 
\[
\mu_E=\absolute{\omega}
\]
on $E(K)$ from $\omega$, cf.\ \cite[Chapter 7.4]{IgusaLocalZeta}.
\newline

Define
\[
E_0(K)=\mathset{P\in E(K)\mid \begin{minipage}{8cm}$P$ reduces mod $\pi$ in the Weierstrass equation for $E$ to a non-singular point
\end{minipage}}\,,
\]
and obtain the following result:

\begin{Proposition}\label{measureEllipticCurve}
It holds true that
\[
\mu_E(E(K))=(E(K):E_0(K))\cdot\frac{\absolute{\overline{E}_0(k)}}{q}
\]
for $K$ a sufficiently large $p$-adic number field, where $\overline{E}_0$ is the (non-singular part of the) identity component of the special fibre $\overline{E}$ over $k$.
\end{Proposition}

\begin{proof}
First, one uses that the quotient $E(K)/E_0(K)$ is a finite abelian group, cf.\ \cite[Corollary 9.2(d)]{Silverman1994}. Then
\'A.\ Weil \cite[Theorem 2.2.5]{WeilAAG} shows that 
\[
\mu_E(E_0(K))=\frac{\absolute{\overline{E}_0(k)}}{\absolute{q}}\,,
\]
which now implies the assertion.
\end{proof}

\begin{remark}
We are not making a claim of originality for the statement of Theorem \ref{measureEllipticCurve}. It is simply needed as a reference for the case of an elliptic curve with a certain type of bad reduction below. And it was observed that this general result follows immediately from the theory of N\'eron models.
\end{remark}

The following special case will now be made explicit: namely, the case of a Tate elliptic curve $E$ over $K$.
An introduction to the theory of Tate elliptic curves can be found e.g.\ in \cite{Silverman1994,FP2004}.
A slightly different approach can be found in the book \cite{FP1981} which, unfortunately, is out of print.
\newline

The $K$-rational points $E_q(K)$ of the Tate elliptic curve 
\[
E_q=\mathds{C}_p/\tilde{q}^{\,\mathds{Z}}
\]
with $\tilde{q}\in K$ such that $0<\absolute{\tilde{q}}<1$ and $\mathds{C}_p$ the field of complex $p$-adic numbers form a $p$-adic analytic $1$-manifold. Diffusion operators  induced by theta functions on $E_{\tilde{q}}(K)$ are introduced in \cite{brad_thetaDiffusionTateCurve}.
There, it was used that any holomorphic $1$-form 
 on $E_{\tilde{q}}$  is nowhere vanishing by Riemann-Roch, and
the measure $\absolute{\omega}$ 
coming from the holomorphic differential
\[
\omega=\frac{dx}{x+y}\,,
\]
where $x,y$ are local coordinates on $E_q$,
is locally of the form 
\[
\absolute{\omega}=\frac{\absolute{dx}}{\absolute{x}}\,,
\]
and is invariant under the action of the abelian group $E_{\tilde{q}}(K)$ on itself,
cf.\ for example \cite[Lemma 2.4]{brad_thetaDiffusionTateCurve}.
The approach here allows to hear its Serre invariant:

\begin{cor}
The wavelet eigenvalue $\lambda_\psi$ associated with the operator $\Delta^s$ on $E(K)$ for a Tate elliptic curve $E$ defined over $K$ satisfies
\[
\lambda_\psi\equiv i(E(K))\equiv 0\mod q-1\,,
\]
where, by definition, $K$ is assumed a sufficently large $p$-adic number field.
\end{cor}

\begin{proof}
This is an immediate consequence of
Theorem \ref{SerreInvariant} and Proposition \ref{measureEllipticCurve}, once one realises that in the case of a Tate curve, the connected components of the special fibre $\mathcal{E}$ are projective lines over $k$, and these have each $q-1$ non-singular points inside $\overline{E}(k)$.
\end{proof}

\begin{remark}
Let us remark that the measure of the set of $K$-rational points of a Tate elliptic curve $E$ over $K$ can be given explicitly as
\[
\mu_E(E(K))=\frac{m(q-1)}{q}\,,
\]
where $\tilde{q}=q^mu$ with $u\in O_K^\times$, and $m\ge 1$ is the number of connected components of the special fibre of the N\'eron model $\mathcal{E}$ of $E$, which in this case  is an $m$-gon of projective lines over $k$.
\end{remark}

\section*{Acknowledgements}

Evgeny Zelenov, David Weisbart and Wilson Z\'u\~niga-Galindo are warmly thanked for fruitful discussions. 
Frank Herrlich and Stefan K\"uhnlein are thanked for a lot of helpful thoughts about the theory of elliptic curves many years ago.
This work is partially supported by the Deutsche Forschungsgemeinschaft under project number 469999674.

\bibliographystyle{plain}
\bibliography{biblio}

\end{document}